\documentclass[a4paper,11pt]{amsart}
\usepackage{cases}

\usepackage{mathrsfs}
\usepackage{amsfonts,amssymb,amscd,amsthm,amsmath,graphicx}
\usepackage{longtable}
\usepackage[all]{xy}

\usepackage{graphicx}

\theoremstyle{plain}
\newtheorem{theorem}{Theorem}[section]
\newtheorem{lemma}[theorem]{Lemma}
\newtheorem{corollary}[theorem]{Corollary}
\newtheorem{proposition}[theorem]{Proposition}

\newtheorem{definition}[theorem]{Definition}
\newtheorem{remark}[theorem]{Remark}
\newtheorem{example}[theorem]{Example}

\newtheorem{question}{Question}[section]

\def\GL{\operatorname{GL}}

\def\SL{\operatorname{SL}}

\def\Coker{\operatorname{Coker}}

\def\deg{\operatorname{deg}}
\def\det{\operatorname{det}}

\def\diag{\operatorname{diag}}

\def\End{\operatorname{End}}

\def\Ext{\operatorname{Ext}}

\def\Gr{\operatorname{Gr}}

\def\Hom{\operatorname{Hom}}

\def\Im{\operatorname{Im}}

\def\ker{\operatorname{ker}}

\def\max{\operatorname{max}}

\def\rad{\operatorname{rad}}
\def\rank{\operatorname{rank}}

\def\span{\operatorname{span}}
\def\Spec{\operatorname{Spec}}

\newcommand{\bbC}{\mathbb{C}}

\newcommand{\bbZ}{\mathbb{Z}}

\begin{document}

\title{Algebraic vector bundles on punctured affine spaces and smooth quadrics}

\author{Brent Doran}
\address[Doran]{Department of Mathematics, HG G 63.2, Ramistrasse 101, 8092 Zurich, Switzerland}
\email{brent.doran@math.ethz.ch}

\author{Jun Yu}
\address[Yu]{School of Mathematics, Institute for Advanced Study, Einstein Drive, Fuld Hall, Princeton,
NJ 08540}
\email{junyu@math.ias.edu}

\abstract{There is a canonical $\mathbb{A}^{n}$-fibration from the $(2n+1)$-dimensional smooth quadric
$Y_{n}$ to the $(n+1)$-dimensional punctured affine space $X_{n}$. For each rank at least $n$, we construct
examples of non-isomorphic algebraic vector bundles on the punctured affine space with isomorphic
pullbacks to the smooth quadric. Moreover, we construct continuous families of arbitrarily large dimensional
pairwise non-isomorphic rank $n$ bundles on $X_{n}$ with trivial pullbacks to the smooth quadric.}
\endabstract



\maketitle

\section{Introduction}

Given a smooth affine variety $X$, denote by $\mathscr{V}_{n}(X)$ the isomorphism classes of rank $n$
algebraic vector bundles on $X$. Morel proved that \footnote[1]{with the rank 2 case completed by Moser}
(cf. \cite{Morel}),  \[\mathscr{V}_{n}(X)=[X,BGL_{n}]_{\mathbb{A}^1}.\] Here , $BGL_{n}$ is the simplicial
classifying space of $\GL_{n}$ (cf. \cite{Morel-Voevodsky}) and $[\cdot,\cdot]_{\mathbb{A}^1}$ denotes
the equivalence classes of maps in the $\mathbb{A}^1$-homotopy category.

The above theorem might make one hope that some form of homotopy invariance holds for the functor
$\mathscr{V}_{n}(X)$ beyond the affine case. Indeed, by the Jouanolou-Thomason homotopy lemma
(cf. \cite{Weibel2}, Proposition 4.4), given a smooth scheme $X$ admitting an ample family of
line bundles (e.g., a quasi-projective variety), there exists a smooth affine scheme $Y$ and
a Zariski locally trivial smooth morphism $f: Y \longrightarrow X$ with fibers isomorphic to affine
spaces ($f$ is called an affine vector bundle torsor).  In particular,
this morphism is an $\mathbb{A}^1$-weak equivalence. Thus the above naive hope would reduce
the study of vector bundles on such schemes to the case of affine varieties. Unfortunately,
Theorem 1.2 in \cite{Asok-Doran} shows that this is false. In \cite{Asok-Doran}, Asok and Doran
constructed continuous families of $\mathbb{A}^{1}$-contractible smooth varieties, which admit
continuous moduli of non-trivial algebraic vector bundles.


Given $n\geq 1$, denote by \[Y_{n}=\SL_{n+1}/\SL_{n}\] and \[X_{n}=\bbC^{n+1}\backslash\{0\},\]
where $\SL_{n}=\SL(n,\bbC)$. Let \[p: Y_{n}\longrightarrow X_{n}\] be the projection which maps an
$(n+1)\times(n+1)$ matrix to its first row (cf. Section \ref{S:e}). Then
$p$ is a fiber bundle with fibers isomorphic to $\bbC^{n}$. Namely, $Y_{n}$ is an
{\it affine vector bundle torsor} over $X_{n}$. From the point of view of the $\mathbb{A}^1$-homotopy
theory, $p$ is an $\mathbb{A}^1$-homotopy weak equivalence. Since the algebraic $K$-theory is
representable in the $\mathbb{A}^1$-homotopy category (cf. Theorem 2.3.1 in \cite{Voevodsky}), it
follows that $p$ induces an isomorphism of algebraic $K$-groups. In particular, we have
$K_{0}(Y_{n})=K_{0}(X_{n})$. Therefore, it is natural to ask the following question:

\begin{question}\label{Q:bundle}
Does there exist a pair of non-isomorphic algebraic vector bundles $(E_1,E_2)$ on $X_{n}$ such
that their pull-backs to $Y_{n}$ satisfy $p^{\ast}(E_1)\cong p^{\ast}(E_2)$?
\end{question}

In this chapter, we give an affirmative answer to Question \ref{Q:bundle}. In fact, given that rank
$r\geq n$, we show that there exist non-trivial bundles on $X_{n}$ of rank $r$ satisfying that their
pull-backs to $Y_{n}$ are trivial (cf. Theorem \ref{T:any rank}). Moreover, we construct continuous
families of pairwise non-isomorphic algebraic bundles on $X_{n}$ of rank $n$ of arbitrarily large
dimension such that their pull-backs to $Y_{n}$ are trivial (cf. Theorem \ref{T:family}). We sketch the
proof as follows. First we define an invariant $e$ for algebraic vector bundles on $\mathbb{P}^{n}$.
Then we use it to study the pull-backs to $X_{n}$ and $Y_{n}$ of bundles on $\mathbb{P}^{n}$. Given an
algebraic vector bundle $E$ on $\mathbb{P}^{n}$, $e(E)$ is a non-negative integer. Most of our studies
are about the bundles $E$ with $e(E)=1$. We also note that Asok and Fasel have made progress reducing
Morel's abstract solution to concrete computation of $\mathbb{A}^1$-homotopy groups of algebraic
spheres (cf. \cite{Asok-Fasel}).

\noindent{\it Notation and conventions.}
Let \[\pi: X_{n}\longrightarrow\mathbb{P}^{n}\] be the natural projection and let \[\rho=\pi\circ p.\]
Let $\mathcal{O}_{\mathbb{P}^{n}}(k)$ be the line bundle on $\mathbb{P}^{n}$ corresponding to the
divisor $kH$, where $H$ is a hyperplane divisor. Given an algebraic vector bundle $E$ on $\mathbb{P}^{n}$,
let $E(k)=E\otimes\mathcal{O}_{\mathbb{P}^{n}}(k)$ and $H^{i}(E)$ denote the $i$-th (sheaf) cohomology
group of $E$.

Throughout, we assume that varieties in this chapter are over the complex field. However, all results
except Swan-Towber's and Swan's theorems and Theorem \ref{T:e=1} (2) hold true for varieties over any
algebraically closed field.

\section{Vector bundles on affine spaces, projective spaces and punctured affine spaces}

The study of algebraic vector bundles on an affine scheme $\Spec A$ is equivalent to the study of
projective $A$-modules. In the 1950s, Serre asked whether algebraic vector bundles on the affine
space $\bbC^{n}$ are trivial. This was addressed by Quillen and Suslin independently in 1976. They
proved that: over a field $k$, any projective $k[x_1,...,x_{n}]$-module of finite rank is trivial.



In the 1950s Grothendieck showed that, any algebraic vector bundle on the projective line $\mathbb{P}^1$
is a direct sum of line bundles $\mathcal{O}_{\mathbb{P}^1}(k)$. In \cite{Schwarzenberger2},
Schwarzenberger introduced {\it almost decomposable bundles}. An algebraic vector bundle on $X=\mathbb{P}^2$
is called almost decomposable if $\dim H^{0}(X,\End E)>1$, where $\End E=E\otimes E^{\ast}$ is the endomorphism
bundle of $E$. Schwarzenberger classified all the almost decomposable bundles on $\mathbb{P}^2$ of rank 2 in
\cite{Schwarzenberger}. 
In \cite{Schwarzenberger2}, Schwarzenberger showed that, any rank 2 bundle $E$ on $X=\mathbb{P}^2$ is of
the form $E=f_{\ast}M$ where $M$ is a line bundle on a non-singular $Y$ and $f: Y\longrightarrow X$ is a
ramified double covering. In the case that $Y=Q_2=\mathbb{P}^1\times\mathbb{P}^1$ or $Y=V_2$, a blowing-up of
$\mathbb{P}^2$ with seven distinct points as base points, Schwarzenberger studied the rank 2 bundles
$E=f_{\ast}M$ intensively. Since the 1970s, the study of algebraic vector bundles on projective spaces has
focused on stable (and semi-stable) vector bundles and their moduli spaces. In the 1970s Barth and Hulek
proved that, the moduli space $M_{\mathbb{P}^{2}}(c_1,c_2)$ of stable vector bundles $E$ on $X=\mathbb{P}^2$
with fixed Chern classes $(c_1(E),c_2(E))=(c_1,c_2)$ is an irreducible, rational and smooth variety (cf.
\cite{Barth} and \cite{Hulek}). The books \cite{Okonek-Schneider-Spindler} and \cite{Huybrechts-Lehn} give
an excellent introduction to the study of vector bundles on projective spaces.

In \cite{Horrocks}, Horrocks studied algebraic vector bundles on the punctured spectrum $Y=\Spec A-\{m\}$ for
a regular integral ring $A$ and an maximal ideal $m$ of it. He defined $\Phi$-equivalence for algebraic vector
bundles on $Y$ and showed that(\cite{Horrocks}, Theorem 7.4), a bundle $\mathcal{E}$ on $Y$ is determined up to
$\Phi$-equivalence by the cohomology modules $H^{i} (\mathcal{E})$ $(0<i<\dim A-1)$ and the extensions
$b^{i}(\mathcal{E})\in\Ext_{A}^2(H^{i+1}(\mathcal{E}),H^{i}(\mathcal{E}))$ ($0<i<\dim A-2$). As an application,
Horrocks got the following beautiful criterion:

\begin{theorem}[Horrocks]\label{T:Horrocks}
An algebraic vector bundle $E$ of rank $r$ on $\mathbb{P}^{n}$ is a direct sum of line bundles
$\mathcal{O}_{\mathbb{P}^n}(k)$ if and only if $H^{i}(E(k))=0$ for any $0<i<n$ and any $k\in\bbZ$, where
$E(k)=E\otimes\mathcal{O}_{\mathbb{P}^n}(k)$.
\end{theorem}

\section{Unimodular and completable rows}

In this section, we give a brief review of the theory of unimodular rows. The readers could refer \cite{Lam}
for more details. Given a commutative ring $R$, let $\GL_{n}(R)$ be the group of $n\times n$ matrices with
entries in $R$. A row vector $\vec{a}=(a_0,a_1,\dots,a_{n})$ is called a {\it unimodular} row if
$Ra_0+Ra_1+\cdots+Ra_{n}=R$. Two unimodular rows $\vec{a}^1,\vec{a}^{2}$ are called equivalent, if
$\vec{a}^2=\vec{a}^1 g$ for some $g\in\GL_{n+1}(R)$. We denote by $\vec{a}^1\sim\vec{a}^{2}$, if
$\vec{a}^1$ and $\vec{a}^{2}$ are equivalent. A unimodular row $\vec{a}$ is called {\it completable} if
$\vec{a}\sim(1,0,\dots,0)$.

Given a unimodular row $\vec{a}$, there is a surjective homomorphism \[R^{n+1}\longrightarrow R,\quad
\sum_{0\leq i\leq n} x_{i}e_{i}\mapsto\sum_{0\leq i\leq n}a_{i}x_{i},\] where $\{e_0,\dots,e_{n}\}$
is the standard basis of $R^{n+1}$. Let $P_{\vec{a}}$ be the kernel of this surjective homomorphism,
which is a projective $R$-module of rank $n$.  It is known that $P_{\vec{a}^1}\cong P_{\vec{a}^2}$
if and only if $\vec{a}^1\sim\vec{a}^2$. In particular, $P_{\vec{a}}$ is trivial if and only if
$\vec{a}$ is completable.

We define the ring $R=\bbC[x_0,x_1,...,x_n,y_0,y_1,...,y_n]/\langle x_0y_0+x_1y_1+\cdots+x_{n}y_{n}-1\rangle$,
and a unimodular row of the form $(f_{0},\dots,f_{n})=(x_0^{a_0},\dots,x_{n}^{a_{n}})$, where
$a_0,\dots,a_{n}$ are non-negative integers. We restate the celebrated $n!$ Theorem of Suslin and its
converse due to Swan-Towber.

\begin{theorem}\label{T:SW-S}
Let $(f_{0},\dots,f_{n})=(x_0^{a_0},\dots,x_{n}^{a_{n}})$, where $a_{i}\in\bbZ_{\geq 1}$.
\begin{itemize}
\item[(1)]{(Suslin, \cite{Lam}) If $n!\mid\prod a_{i}$, then $(f_0,f_1,...,f_{n})\in R^{n+1}$
is completable.}
\item[(2)]{(Swan-Towber, \cite{Swan-Towber}) If $n!\nmid\prod a_{i}$, then
$(f_0,f_1,...,f_{n})\in R^{n+1}$ is not completable.}
\end{itemize}
\end{theorem}

We remark that Suslin's theorem holds on any field.

Given a unimodular row $(f_{0},\dots,f_{n})$ with $f_{i}$ a homogeneous polynomial of degree $a_{i}$,
the following generalization of Theorem \ref{T:SW-S} states that:

\begin{theorem}[\cite{Kumar} and \cite{Swan2}]\label{T:Kumar-Swan}
Let $\{f_{0},\dots,f_{n}\}$ be homogeneous polynomials of $x_0,...,x_{n}$ satisfying that
\[\rad(f_0,f_1,...,f_{n})=(x_0,x_1,...,x_{n}).\]
\begin{itemize}
\item[1),]{(Mohan Kumar, \cite{Kumar}) If $n!\mid\prod\deg f_{i}$, then $(f_0,f_1,...,f_{n})\in R^{n+1}$
is completable.}
\item[2),]{(Swan, \cite{Swan2}) If $n!\nmid\prod\deg f_{i}$, then $(f_0,f_1,...,f_{n})\in R^{n+1}$ is
not completable.}
\end{itemize}
\end{theorem}

Mohan Kumar's theorem holds for any algebraic closed field.

\section{Invariant $e$ and some examples}\label{S:e}

Let $Y_{n}=\SL_{n+1}/\SL_{n}$ and $X_{n}=\bbC^{n+1}\backslash\{0\}$, and let $\mathbb{P}^{n}$ be
$n$-dimensional projective space over $\bbC$. Here $\SL_{n}$ is included in $\SL_{n+1}$ in the
following way \[\SL_n=\{\diag\{1,B\}\in\SL_{n+1}|B\in\SL_{n}\}.\]
Let $p:Y_{n}\rightarrow X_{n}$ be the map $p([A])=(a_{00},\cdots,a_{0,n})$ (the first row) for any
\[A=(a_{i,j})_{0\leq i,j\leq n}\in\SL_{n+1},\] where \[[A]=\SL_{n}A\in\SL_{n+1}/\SL_{n}=Y_{n}.\]
Let $\pi:X_{n}\rightarrow\mathbb{P}^{n}$ be the natural projection, and
\[\rho=\pi\circ p:Y_{n}\rightarrow\mathbb{P}^{n}.\]

For $A=(a_{i,j})_{0\leq i,j\leq n}\in\SL_{n+1}$, let $x_{i}=a_{0,i}$ and \[y_{i}=(-1)^{i}\det A_{i}\]
where $A_{i}$ is the $n\times n$ sub-matrix of $A$ with the first row (indexed by 0) and the $(i+1)$-th
column (indexed by $i$) deleted. Then we have the isomorphism
\[\SL_{n+1}/\SL_{n}\cong\{(x_0,\cdots,x_{n},y_0,\cdots,y_{n})\in\bbC^{2n+2}:
\sum_{0\leq i\leq n}x_{i}y_{i}=1\}.\] The last equality implies that $Y_{n}$ is isomorphic to the
$(2n+1)$-dimensional smooth quadric.

Let \[R=R[Y_{n}]=\bbC[x_0,\cdots,x_{n},y_0,\cdots,y_{n}]/\langle\sum_{0\leq i\leq n}x_{i}y_{i}-1
\rangle\] and $S=R[X_{n}]=\bbC[x_0,x_1,\cdots,x_{n}]$. Let $m=(x_0,x_1,\cdots,x_{n})$ be the maximal
ideal corresponding to the point ${0}=\Spec S-X_{n}$.

Given a vector bundle $E$ of rank $r$ on $\mathbb{P}^{n}$ and $i\in\bbZ$, define \[M_{i}(E)=
\bigoplus_{k\in\bbZ}H^{i}(E(k)),\] where $E(k)=E\otimes\mathcal{O}_{\mathbb{P}^{n}}(k)$. Let
$M_{E}=M_0(E)$ and $M'_{E}=M_{E}/mM_{E}$.

\begin{proposition}\label{P:bundle-basic}
Given an algebraic vector bundle $E$ of rank $r$ on $\mathbb{P}^{n}$,
\begin{itemize}
\item[(1)]{if $i<0$ or $i>n$, then $M_{i}(E)=0$.}
\item[(2)]{For any $i$ with $0<i<n$, $M_{i}(E)$ is of finite dimension.}
\item[(3)]{$M_{E}$ is a finitely generated $S=\bigoplus_{k\in\bbZ}\mathcal{O}_{\mathbb{P}^{n}}(k)$
module.}
\item[(4)]{$M'_{E}$ is of finite dimension.}
\end{itemize}
\end{proposition}

\begin{proof}
For $(1)$, if $i<0$, it is clear that $H^{i}(E(k))=0$ for any integer $k$. Thus $M_{i}(E)=0$. If
$i>n$, for any $k\in\bbZ$, by the Serre duality $H^{i}(E(k))=(H^{n-i}(E(-n-1-k)))^{\ast}=0$. Hence
$M_{i}(E)=0$.

For $(2)$, since $\mathcal{O}(1)=\mathcal{O}_{\mathbb{P}^{n}}(1)$ is very ample, given an $i>0$ there
exists a $k_{i}\in\bbZ$ such that $H^{i}(E(k))=0$ for $k\geq k_{i}$ (cf. \cite{Hartshorne},
Theorem 5.2). By the Serre duality, for any $i<n$, $H^{i}(E(k))=(H^{n-i}(E^{\ast}(-n-1-k)))^{\ast}$.
Thus there exists a $k'_{i}\in\bbZ$ such that $H^{i}(E(k))=0$ for $k\leq k'_{i}$. Since each
$H^{i}(E(k))$ is of finite-dimension, thus $M_{i}(E)$ is of finite dimension.

For $(3)$, it follows from \cite{Huybrechts-Lehn}, Lemma 1.7.2 and \cite{Hartshorne}, Theorem 5.2 that,
$M_{E}$ is a finitely generated $S=\bigoplus_{k\in\bbZ}\mathcal{O}_{\mathbb{P}^{n}}(k)$ module.

Finally, $(4)$ follows from $(3)$.
\end{proof}

\begin{definition}\label{D:e}
Given an algebraic vector bundle $E$ of rank $r$ on $\mathbb{P}^n$, define $e(E)+r$ to be the minimal
number of generators of the $S$-module $M_{E}$.
\end{definition}

\begin{lemma}\label{L:3-term resolution}
Given an algebraic vector bundle $E$ of rank $r$ on $\mathbb{P}^n$, we have $e=e(E)=\dim_{\bbC}(M'_{E})-r$.
Moreover, there exists an exact sequence
\[0\longrightarrow E'\longrightarrow\bigoplus_{1\leq i\leq e+r}\mathcal{O}_{\mathbb{P}^{n}}(a_{i})
\longrightarrow E\longrightarrow 0\] for some integers $a_1,\dots,a_{e+r}$ and a vector bundle
$E'$ of rank $e$. In particular we have $e(E)\geq 0$.
\end{lemma}

\begin{proof}
If the elements $x_1,x_2,\dots,x_{e+r}$ generate the $S$-module $M_{E}$, then they generate the $S/m=\bbC$
vector space $M'_{E}=M_{E}/mM_{E}$. Thus $\dim_{\bbC}(M'_{E})\leq e(E)+r$. On the other hand, by Proposition
\ref{P:bundle-basic}, $M_{E}$ is a finitely generated $S$-module. Thus there exists $k_0\in\bbZ$ such
that $H^{0}(E(k))=0$ for $k<k_0$. Let $x_1,\dots,x_{t}$ be homogeneous elements of $M_{E}$ such that,
\[[x_1]=x_1+mM_{E},\dots,[x_{t}]=x_{t}+mM_{E}\] is a basis of $M'_{E}$. From the condition of $H^{0}(E(k))=0$
for $k<k_0$, one can show that $x_1,\dots,x_{t}$ generate the $S$-module $M_{E}$. Hence
$e(E)+r\leq\dim_{\bbC}(M'_{E})$. Therefore $e(E)=\dim_{\bbC}(M'_{E})-r$.

By the argument above, there exists a system of homogeneous generators
\[\{x_i\in H^{0}(E(-a_{i}))|\ 1\leq i\leq e+r\}\] of the $S$-module $M_{E}$. For each $i$, the element
${x_i\in H^{0}(E(-a_{i}))}$ gives a map \[]\mathcal{O}_{\mathbb{P}^{n}}(a_{i})\longrightarrow E.\] Summing them
up, we get a map of vector bundles \[\psi: \bigoplus_{1\leq i\leq e+r}\mathcal{O}_{\mathbb{P}^{n}}(a_{i})
\longrightarrow E.\] By Theorem 5.17 in \cite{Hartshorne} , $\psi$ is a surjective map of vector bundles.
Let $E'$ be the kernel of $\psi$. Then we get an exact sequence \[0\longrightarrow E'\longrightarrow
\bigoplus_{1\leq i\leq e+r}\mathcal{O}_{\mathbb{P}^{n}}(a_{i})\longrightarrow E\longrightarrow 0.\]
It is clear that $\rank E'=(e(E)+r)-\rank E=e(E)$. Therefore $e(E)\geq 0$.
\end{proof}

\begin{corollary}\label{C:e-additive}
Given an algebraic vector bundle $E$ of rank $r$ on $\mathbb{P}^n$, $e(E)\geq 0$.

Given two algebraic vector bundles $E_1,E_2$ of finite rank on $\mathbb{P}^n$,
$e(E_1\oplus E_2)=e(E_1)+e(E_2)$.
\end{corollary}

\begin{proof}
The first statement follows from the second statement in Lemma \ref{L:3-term resolution}.
The second statement follows from the first statement in Lemma \ref{L:3-term resolution}.
\end{proof}

\begin{proposition}\label{P2:e=1}
Given an algebraic vector bundle $E$ on $\mathbb{P}^{n}$ with $e(E)=1$, there exists an exact
sequence \[0\longrightarrow\mathcal{O}_{\mathbb{P}^{n}}(b)\longrightarrow\bigoplus_{1\leq i\leq r+1}
\mathcal{O}_{\mathbb{P}^{n}}(a_{i})\longrightarrow E\longrightarrow 0\] for some integers
$b,a_1,\dots,a_{r}$.
\end{proposition}

\begin{proof}
This follows from Lemma \ref{L:3-term resolution}. Note that, since $\rank E'=e(E)=1$, we have
$E'\cong\mathcal{O}_{\mathbb{P}^{n}}(b)$ for some integer $b$.
\end{proof}

\begin{theorem}\label{T:pull-back}
Given an algebraic vector bundle $E$ of rank $r$ on $\mathbb{P}^n$, the following conditions are equivalent
to each other:
\begin{itemize}
 \item[(1)]{$e(E)=0$.}
 \item[(2)]{$E$ is a direct sum of line bundles.}
 \item[(3)]{$\pi^{\ast}E$ is a trivial bundle on $X$.}
 \item[(4)]{$H^{i}(\pi^{\ast}E)=0$ for any $1\leq i\leq n-1$.}
 \item[(5)]{$H^{i}(E(k))=0$ for any $1\leq i\leq n-1$ and any $k\in\bbZ$,}
\end{itemize}
\end{theorem}

\begin{proof}
Lemma \ref{L:3-term resolution} implies that $(1)\Leftrightarrow(2)$. It follows from Theorem 2.3.1 in
\cite{Okonek-Schneider-Spindler} that $(5)\Leftrightarrow(2)$. It is obvious that $(2)\Rightarrow(3)$ and
$(3)\Rightarrow(4)$.

To show $(4)\Leftrightarrow(5)$, we consider \[\pi: X_{n}=\bbC^{n+1}-\{0\}\longrightarrow\mathbb{P}^{n},\]
which is a fibre bundle with fibers all isomorphic to $\bbC^{\ast}=\bbC-\{0\}$. A straightforward
calculation shows that $\pi_{\ast}(\pi^{\ast}E)\cong\bigoplus_{k\in\bbZ}E(k)$. Since $\pi$ is an affine
morphism, by the Leray spectral sequence we get \[H^{i}(\pi^{\ast}E)=H^{i}(\pi_{\ast}(\pi^{\ast}E))=
\bigoplus_{k\in\bbZ}H^{i}(E(k)).\] From this equality, we get $(4)\Leftrightarrow(5)$.
\end{proof}

\begin{proposition}\label{P:3-term resolution}
Let $E$ be a vector bundle of rank $r$ on $\mathbb{P}^n$ with $M_{i}(E)=0$ for $1\leq i\leq n-2$.
Then we have an exact sequence
\[0\longrightarrow\bigoplus_{1\leq i\leq e}\mathcal{O}_{\mathbb{P}^{n}}(b_{i})\longrightarrow
\bigoplus_{1\leq i\leq r+e}\mathcal{O}_{\mathbb{P}^{n}}(a_{i})\longrightarrow E\longrightarrow 0\]
for some integers $b_1,\dots,b_{e},a_1,\dots,a_{e+r}$.
\end{proposition}

\begin{proof}
By Lemma \ref{L:3-term resolution}, we have an exact sequence \[0\longrightarrow E'\longrightarrow
\bigoplus_{1\leq i\leq e+r}\mathcal{O}_{\mathbb{P}^{n}}(a_{i})\longrightarrow E\longrightarrow 0\] for
some integers $a_1,\dots,a_{e+r}$ and a vector bundle $E'$ of rank $e$. Let \[E''=\bigoplus_{1\leq i\leq e+r}
\mathcal{O}_{\mathbb{P}^{n}}(a_{i})\] and denote by $\psi:E''\longrightarrow E$ the second map. By
the proof of Lemma \ref{L:3-term resolution}, the map $H^{0}(E''(k))\longrightarrow H^{0}(E(k))$ is
surjective for $k\in\bbZ$.

From the short exact sequence $0\longrightarrow E'\longrightarrow E''\longrightarrow E\longrightarrow 0$,
we get \[H^{i-1}(E''(k))\longrightarrow H^{i-1}(E(k))\longrightarrow H^{i}(E'(k))\longrightarrow
H^{i}(E''(k)).\] For $2\leq i\leq n-1$, we have $H^{i-1}(E''(k))=H^{i}(E''(k))=0$, since $E''(k)$ is a
direct sum of line bundles. By the assumption of $H^{i-1}(E(k))=0$, we get $H^{i}(E'(k))=0$. For $i=1$,
$H^{1}(E''(k))=0$, since $E''(k)$ is a direct sum of line bundles and the map $H^{0}(E''(k))
\longrightarrow H^{0}(E(k))$ is surjective by the proof of Lemma \ref{L:3-term resolution}. Therefore,
$H^{1}(E'(k))=0$.

By these, we get $H^{i}(E(k))=0$ for $1\leq i\leq n-1$ and $k\in\bbZ$. By Theorem \ref{T:pull-back}, $E'$
is a direct sum of line bundles.
\end{proof}

\begin{theorem}\label{T:easy examples}
Let $E$ be a vector bundle of rank $r$ on $\mathbb{P}^n$ with $e(E)>0$. Then,
\begin{itemize}
\item[(1)]{$\pi^{\ast}E\oplus(\mathcal{O}_{X_{n}})^{t}$ is not trivial for $t\geq 0$.}
\item[(2)]{If $M_{i}(E)=0$ for $1\leq i\leq n-2$, then $\rho^{\ast}E\oplus(\mathcal{O}_{Y_{n}})^{t}$ is
trivial for $t\geq e$.}
\end{itemize}
\end{theorem}

\begin{proof}
For (1), note that $\pi^{\ast}E\oplus (\mathcal{O}_{X_{n}})^{t}=\pi^{\ast}(E\oplus
(\mathcal{O}_{\mathbb{P}^{n}})^{t})$ and $e(E\oplus(\mathcal{O}_{\mathbb{P}^{n}})^{t})=e(E)>0$. Hence
$\pi^{\ast}E\oplus(\mathcal{O}_{X_{n}})^{t}$ is not trivial by Theorem \ref{T:pull-back}.

For (2), by Proposition \ref{P:3-term resolution} we have the following exact sequence
\[0\longrightarrow\bigoplus_{1\leq i\leq e}\mathcal{O}_{\mathbb{P}^{n}}(b_{i})\longrightarrow
\bigoplus_{1\leq i\leq r+e}\mathcal{O}_{\mathbb{P}^{n}}(a_{i})\longrightarrow E\longrightarrow 0.\]
Since $\rho: Y_{n}\rightarrow\mathbb{P}^{n}$ is an affine morphism, pulling back to $Y_{n}$, we get an
exact sequence \[0\longrightarrow(\mathcal{O}_{Y_{n}})^{e}\longrightarrow(\mathcal{O}_{Y_{n}})^{e+r}
\longrightarrow\rho^{\ast}E\longrightarrow 0.\] Since $Y_{n}$
is affine, the above short exact sequence of vector bundles splits. Thus \[\rho^{\ast}E\oplus
(\mathcal{O}_{Y_{n}})^{t}\cong(\rho^{\ast}E\oplus(\mathcal{O}_{Y_{n}})^{e})\oplus
(\mathcal{O}_{Y_{n}})^{t-e}\cong(\mathcal{O}_{Y_{n}})^{e+r}\oplus(\mathcal{O}_{Y_{n}})^{t-e}
\cong(\mathcal{O}_{Y_{n}})^{t+r}\] is trivial, if $t\geq e$. \end{proof}

Note that, the assumption of $M_{i}(E)=0$ for $1\leq i\leq n-2$ holds true if $n=2$.

\begin{example}[An explicit and simple example]\label{E:first example}
Let $G=\SL_3$,
\[P=\left\{\left(\begin{array}{cc}\lambda&\alpha^{t}\\0_2&B\end{array}\right)|\ B\in\GL_2,\lambda\det B=1,
\alpha\in\bbC^{2}\right\},\] \[P'=\left\{\left(\begin{array}{cc}1&\alpha^{t}\\0_2&B\end{array}\right)|\
B\in\SL_2,\alpha\in\bbC^{2}\right\}\] and \[H=\left\{\left(\begin{array}{cc}1&0_{2}^{t}\\0_2&B\end{array}
\right)|\ B\in\SL_2\right\}.\] Let $\mathfrak{g}$, $\mathfrak{p}$, $\mathfrak{p}'$ and $\mathfrak{h}$ be
their Lie algebras. We have the following identifications \[Y_{2}=G/H,\] \[X_{2}=G/P'\] and
\[\mathbb{P}^2=G/P.\] Let \[E_1=G\times_{P}(\mathfrak{g}/\mathfrak{p}')\] and
\[E_2=G\times_{P}(\mathfrak{g}/\mathfrak{p}\oplus\mathfrak{p}/\mathfrak{p}').\]
They are algebraic vector bundles on $\mathbb{P}^2=G/P$ of rank $3$. Then $\pi^{\ast}(E_1)$ and
$\pi^{\ast}(E_2)$ are non-isomorphic on $X_2$. However, their pull-backs to $Y_2$ are isomorphic.
\end{example}

\begin{proof}
$\pi^{\ast}(E_1)=G\times_{P'}(\mathfrak{g}/\mathfrak{p}')$ is the tangent bundle of $G/P'\cong
\bbC^{3}\backslash\{0\}$. The tangent bundle of $\bbC^{3}\backslash\{0\}$ is clearly trivial, thus
$\pi^{\ast}(E_1)\cong(\mathcal{\mathcal{O}}_{X_2})^3$.
By the Euler sequence (cf. Page $6$ in\cite{Okonek-Schneider-Spindler}) \[0\longrightarrow
\mathcal{O}_{\mathbb{P}^{n}}(-1)\longrightarrow(\mathcal{O}_{\mathbb{P}^{n}})^{n+1}\longrightarrow
T_{\mathbb{P}^{n}}(-1)\longrightarrow 0,\] we get $e(T_{\mathbb{P}^{n}})=1$. In the case of $n=2$, we
get $e(E_2)=e(T_{\mathbb{P}^2})=1$. Hence $\pi^{\ast}(E_2)$ is not trivial. Therefore
$\pi^{\ast}(E_1)\not\cong\pi^{\ast}(E_2)$.

Since $H$ is semisimple, $\mathfrak{g}/\mathfrak{p}'\cong\mathfrak{g}/\mathfrak{p}\oplus
\mathfrak{p}/\mathfrak{p}'$ as $H$-modules. Hence $\rho^{\ast}(E_1)\cong\rho^{\ast}(E_2)$.

\end{proof}



\section{Bundles with $e=1$ and more examples}

Let $E$ be a rank $r$ vector bundle on $\mathbb{P}^n$ with $e(E)=1$. By Proposition \ref{P2:e=1},
we have an exact sequence \[0\longrightarrow E'=\mathcal{O}_{\mathbb{P}^{n}}(b)\longrightarrow E''=
\bigoplus_{0\leq i\leq r}\mathcal{O}_{\mathbb{P}^{n}}(b_i)\longrightarrow E\longrightarrow 0\] such that
the map $H^{0}(E''(k))\longrightarrow H^{0}(E(k))$ is surjective for $k\in\bbZ$. One can show that the
bundle $E$ determines the numbers $b,b_0,...,b_r$. Moreover we have $b<\min\{b_0,b_1,...,b_r\}$. The reason
is that: if $b=\min\{b_0,b_1,...,b_r\}$, then $E$ is a direct sum of line bundles, which contradicts
to the fact that $e(E)=1$. Moreover, the maps $\phi: E'\longrightarrow E''$ and $\psi: E''\longrightarrow E$
are determined by $E$ up to linear changes of bases.

Let $a_{i}=b_{i}-b>0$ and $f_{i}\in H^{0}(\mathcal{O}_{\mathbb{P}^{n}}(a_i))\cong\Hom(\mathcal{O}_{\mathbb{P}^n}(b),
\mathcal{O}_{\mathbb{P}^n}(b_i))$. Then for an algebraic vector bundle $E$ on $\mathbb{P}^{n}$ with rank $r$
and $e(E)=1$, we associate $E$ with the integer $b$, positive integers $a_0,\dots,a_{r}$ and homogeneous
polynomials $f_{0},f_1,\dots,f_{r}\in\bbC[x_0,x_1,...,x_{n}]$ of degree $a_{0},a_1,\dots,a_{r}$,
respectively.

\begin{proposition}\label{P:e=1}
Given $a_{i}>0$ and $f_{i}\in H^{0}(\mathcal{O}_{\mathbb{P}^{n}}(a_i))$, $0\leq i\leq r$, the following
conditions are equivalent to each other: \begin{itemize}
\item[(1)]{The map $\phi=(f_0,...,f_r): \mathcal{O}_{\mathbb{P}^n}\longrightarrow \bigoplus_{0\leq i\leq r}
O_{\mathbb{P}^n}(a_i)$ is an injective map of algebraic vector bundles.}
\item[(2)]{The zero locus of $(f_0,...,f_r)$ in $\mathbb{P}^n$ is empty.}
\item[(3)]{$\rad(f_0,f_1,...,f_{r})=(x_0,x_1,...,x_{n})$ in $\bbC[x_0,...,x_{n}]$.}
\end{itemize}
In the case that these conditions are satisfied, we have that $r\geq n$ and $E=\Coker\phi$ is a bundle of rank $r$
with $e(E)=1$.
\end{proposition}

\begin{proof}
Since the locus of the points in $\mathbb{P}^{n}$ where $\phi$ is not injective is equal to the zero locus
of $(f_0,...,f_r)$, $(1)$ is equivalent to $(2)$. It is well-known that $(2)\Leftrightarrow(3)$.

In the case that $(1)-(3)$ hold, $r\geq n$ since the Krull dimension of $\bbC[x_0,...,x_{n}]$ is $n+1$. It is clear
that $\Coker\phi$ is an algebraic vector bundle of rank $r$ .

Write $E'=\mathcal{O}_{\mathbb{P}^n}$, $E''=\bigoplus_{0\leq i\leq r}\mathcal{O}_{\mathbb{P}^n}(a_i)$ and
$E=\Coker\phi$. By the short exact sequence $0\longrightarrow E'\longrightarrow E''\longrightarrow
E\longrightarrow 0$ and $H^{1}(E'(k))=0$ for $k\in\bbZ$, we get a short exact sequence
\[0\longrightarrow H^{0}(E'(k))\longrightarrow H^{0}(E''(k))\longrightarrow H^{0}(E(k))
\longrightarrow 0.\] Since $a_{i}>0$ for $0\leq i\leq r$, $\bigoplus_{k<0}H^{0}(E''(k))$ generates
$\bigoplus_{k\in\bbZ}H^{0}(E''(k))$ as an $S$-module. Thus $\bigoplus_{k<0}H^{0}(E(k))$ generates
$\bigoplus_{k\in\bbZ}H^{0}(E(k))$. As $H^{0}(E'(k))=0$ for $k<0$, $H^{0}(E''(k))\longrightarrow
H^{0}(E(k))$ is an isomorphism for $k<0$. Hence $\bigoplus_{k\in\bbZ}H^{0}(E(k))$ and
$\bigoplus_{k\in\bbZ}H^{0}(E''(k))$ have the same minimal number of generators (cf. Lemma
\ref{L:3-term resolution} and its proof), which is equal to $r+1$. Therefore $e(E)=(r+1)-r=1$.
\end{proof}

\begin{lemma}\label{L:Claim1}
If $\rank E=r\geq n+1$ and $e(E)=1$, then there exist $r+1$ homogeneous generators $f_0,f_1,\dots,f_{r}$
of $M_{E}$ such that the zero locus of $$(f_0,f_1,\dots,f_{i_0-1},f_{i_0+1},\dots,f_{r-1},f_{r})$$ is still
empty for some $i_0$.
\end{lemma}

\begin{proof}
Suppose we can not find a system of $r+1$ generators such that the zero locus of
\[](f_0,f_1,\dots,f_{i_0-1},f_{i_0+1},\dots,f_{r-1},f_{r})\] is still empty for some $i_0$.
We prove by induction that, for any system of $r+1$ generators $(f_0,f_1,\dots,f_{r})$ of $M_{E}$,
any $1\leq k\leq n+1$ and any subset $I\subset\{0,1,\dots,r\}$ of cardinality $k$, the zero locus of
$\{f_{i}|\ i\in\{0,1,\dots,r\}-I\}$ has dimension $k-1$.

For $k=1$, by assumption the zero locus $Z_{i}$ of $(f_0,f_1,\dots,f_{i-1},f_{i+1},\dots,f_{r-1},f_{r})$
is non-empty for each $i$. However, $Z_{i}\cap Z(f_{i})=\emptyset$ by the unimodularity condition. Thus
$\dim Z_{i}\leq 0$. Hence $\dim Z_{i}=0$. This proves the assertion in the case of $k=1$.

For $1\leq k\leq n$, suppose the assertion holds for $1,2,\dots,k$. We prove that the zero locus
$Z_{0,1,\dots,k}$ of $(f_{k+1},\dots,f_{r})$ has dimension $k$.
By the induction hypothesis, $Z_{0,1,\dots,k}\cap Z(f_{k})=Z_{0,1,\dots,k-1}$ has dimension $k-1$.
Since $Z(f_{k})$ is a hypersurface, we get $\dim Z_{0,1,\dots,k}\leq k$ by the Projective Dimension
Theorem (cf. \cite{Hartshorne}, Theorem 7.2). On the other hand, we have $\dim Z_{0,1,\dots,k}\geq
\dim Z_{0,1,\dots,k-1}=k-1$. Thus $\dim Z_{0,1,\dots,k}=k-1$ or $k$. If $\dim Z_{0,1,\dots,k}\neq k$,
then $\dim Z_{0,1,\dots,k}=k-1$. Hence $Z_{0,1,\dots,k}$ is a $(k-1)$-dimensional closed subset of
$\mathbb{P}^{n}$. Let $C_1,\dots,C_{s}$ be all $(k-1)$-dimensional irreducible components of
$Z_{0,1,\dots,k}$. Without loss of generality we may assume that $\deg f_{k-1}\leq\deg f_{k}$. Let
$l=\deg f_{k}-\deg f_{k-1}\geq 0$. Choose a linear function $g$ nonvanishing on each $C_{j}$,
$1\leq j\leq s$. Since $Z_{0,1,\dots,k}\cap Z(f_{k})\cap Z(f_{k-1})=Z_{0,1,\dots,k-2}$ has dimension
$k-2<k-1$, we have $C_j\not\subset Z(f_{k})$ or $C_j\not\subset Z(f_{k-1})$ for each $1\leq j\leq s$.
Thus at least one of $f_{k}$ and $f_{k-1}$ does not vanish on $C_{j}$. Since $g$ does not vanish on
$C_{j}$ as well and $C_{j}$ is irreducible, at least one of $f_{k}$ and $g^{l}f_{k-1}$ does not vanish
on $C_{j}$. Write $S_{j}$ for the set of points $[\lambda_0,\lambda_1]\in\mathbb{P}^{1}$ such that
$\lambda_{0}f_{k}+\lambda_{1}f_{k-1}g^{l}$ vanishes on $C_{j}$. Hence each $S_{j}$ has at most one
point. Choose a point $[1,\lambda_1]\in\mathbb{P}^{1}$ not lying in $\cup_{1\leq k\leq s}S_{j}$.
Let \[f'_{k}=f_{k}+\lambda_{1}f_{k-1}g^{l}.\] Then $f'_{k}$ does not vanish on each $C_{j}$,
$1\leq j\leq s$. That is to say, $C_{j}\not\subset Z(f'_{k})$. Hence
\[\dim (C_{j}\cap Z(f'_{k}))<\dim C_{j}=k-1,\] since $C_{j}$ is irreducible. Moreover, we get
\[\dim(Z_{0,1,\dots,k}\cap Z(f'_{k}))\leq k-2.\] However, \[f_0,\dots,f_{k-1},f'_{k}=f_{k}+
\lambda_{1}f_{k-1}g^{l},f_{k+1},\dots,f_{r}\] is also a system of generators of $M_{E}$.
Hence \[\dim(Z_{0,1,\dots,k}\cap Z(f'_{k}))\leq k-2\] contradicts the induction hypothesis.
Therefore $Z_{0,1,\dots,k}$ has dimension $k$.

Similarly, we can show that the zero locus $Z_{i_0,i_1,\dots,i_{k}}$ of
\[\{f_{i}|\ i\in\{0,1,\dots,r\}-\{i_0,i_1,\dots,i_{k}\}\}\] has dimension $k$ for each set of
indices $0\leq i_0<\cdots<i_{k}\leq r$. In this way we conclude the induction step.

Let $k=n+1\leq r$. Then the zero locus of $f_{0},\dots,f_{r-n-1}$ has dimension $k-1\geq n$. It is
impossible as it is a proper closed subset of $\mathbb{P}^{n}$. Therefore we reach the conclusion of
the Lemma.
\end{proof}

\begin{theorem}\label{T2:e=1}
Let $E$ be an algebraic vector bundle on $\mathbb{P}^n$ with $e(E)=1$. If $\rank E=r\geq n+1$, then
$\rho^{\ast}E$ is trivial.
\end{theorem}

\begin{proof}
By Propositions \ref{P2:e=1} and \ref{P:e=1}, $E$ is determined by the integer $b$, $r+1$ positive
integers $a_0,\dots,a_{r}$ ($a_0\leq a_1\leq\cdots\leq a_{r-1}\leq a_{r}$) and homogeneous polynomials
$f_0,\dots,f_{r}$ with $f_{i}$ of degree $a_{i}$, $0\leq i\leq r$. These homogeneous polynomials
$f_0,\dots,f_{r}$ arise from a choice of generators of $M_{E}=\bigoplus_{k\in\bbZ}H^{0}(E(k))$, and they
are not uniquely determined by $E$. For example, if \[f'_{i}=f_i+\sum_{0\leq j\leq i-1}g_{i,j}f_{j},\
0\leq i\leq r,\] with $g_{i,j}$ a homogeneous polynomial of degree $a_{i}-a_{j}$, then
$f'_0,\dots,f'_{r}$ and $f_0,\dots,f_{r}$ define isomorphic vector bundles.

By Lemma \ref{L:Claim1}, we can choose $f_0,f_1,\dots,f_{r}$ such that the zero locus of
\[(f_0,f_1,\dots,f_{i_0-1},f_{i_0+1},\dots,f_{r-1},f_{r})\] is empty for some $0\leq i_0\leq r$.
We have an exact sequence \[0\longrightarrow E'=\mathcal{O}_{\mathbb{P}^{n}}\stackrel{\phi}
\longrightarrow E''=\bigoplus_{0\leq i\leq r}\mathcal{O}_{\mathbb{P}^{n}}(a_i)\stackrel{\psi}
\longrightarrow E(-b)\longrightarrow 0.\] Write \[E_1=\bigoplus_{0\leq i\leq r,i\neq i_0}
\mathcal{O}_{\mathbb{P}^{n}}(a_i)\] and \[E_2=\mathcal{O}_{\mathbb{P}^{n}}(a_{i_0}).\] Let
$p:E''\longrightarrow E_1$ be the natural projection and $\phi'=p\circ\phi$.
Since the zero locus of \[](f_0,f_1,\dots,f_{i_0-1},f_{i_0+1},\dots,f_{r-1},f_{r})\] is empty,
$\phi'=p\circ\phi$ is an injective map of vector bundles. Write $E_3=\Coker\phi'$. We have an exact
sequence \[0\longrightarrow E_2\longrightarrow E(-b)\longrightarrow E_3\longrightarrow 0.\]
Thus, we get \begin{eqnarray*}&&\rho^{\ast}E\\&\cong&\rho^{\ast}(E_2)\oplus\rho^{\ast}(E_3)\
(\textrm{since } Y \textrm{ is affine})\\&\cong&\mathcal{O}_{Y}\oplus\rho^{\ast}(E_3)\ (\textrm{since }
E_2\textrm{ is a line bundle})\\&\cong&(\mathcal{O}_{Y})^{r}\ (\textrm{since } e(E_3)=1).\end{eqnarray*}
\end{proof}

\begin{theorem}\label{T:e=1}
Given an algebraic vector bundle $E$ on $P=\mathbb{P}^n$ of rank $n$ and with $e(E)=1$, let
$a_0,a_1,\dots,a_{n}$ be the degrees of homogeneous polynomials defining $E$.
\begin{itemize}
\item[(1)]{If $n!\mid a_0a_1\cdots a_{n}$, then $\rho^{\ast}E$ is trivial.}
\item[(2)]{If $n!\nmid a_0a_1\cdots a_{n}$, then $\rho^{\ast}E$ is not trivial.}
\end{itemize}
\end{theorem}
\begin{proof}
This follows from Proposition \ref{P2:e=1}, Proposition \ref{P:e=1} and Theorem \ref{T:Kumar-Swan}.
\end{proof}

\begin{theorem}\label{T:any rank}
Given an integer $r\geq n$, there exists an algebraic vector bundle $E$ on $P=\mathbb{P}^{n}$ of
rank $r$ with $\pi^{\ast}E$ non-trivial and $\rho^{\ast}E$ trivial.
\end{theorem}
\begin{proof}
This follows from Theorem \ref{T:pull-back}, Proposition \ref{P:e=1} and Theorem \ref{T:e=1}.
\end{proof}

\begin{lemma}\label{L:indecomposable}
Let $E$ be an algebraic vector bundle on $\mathbb{P}^{n}$ with $e(E)=1$. If $\rank E=n$, then
$E$ is indecomposable.
\end{lemma}
\begin{proof}
Suppose that $E$ is decomposable, i.e., $E=E_1\oplus E_2$ where $E_1$ and $E_2$ are non-zero algebraic
vector bundles on $\mathbb{P}^{n}$. Since $1=e(E)=e(E_1)+e(E_2)$, we have $e(E_1)=1$ or $e(E_2)=1$. We
may assume that $e(E_1)=1$. By Proposition \ref{P:e=1}, we have $\rank E_1\geq n$. Thus
$\rank E=\rank E_1+\rank E_2\geq n+1$, which contradicts to the fact that $\rank E=n$.
\end{proof}

\begin{theorem}[\cite{Horrocks}, Proposition 9.5]\label{T:Horrocks2}
Let $E_1,E_2$ be two indecomposable algebraic vector bundles on $\mathbb{P}^{n}$ of the same rank.
If $\pi^{\ast}(E_1)\cong\pi^{\ast}(E_2)$, then $E_2\cong E_1(k)$ for some $k\in\bbZ$.
\end{theorem}

\begin{theorem}\label{T:family}
There exist continuous families of arbitrarily large dimension of rank $n$ algebraic vector bundles on
$X_{n}$ satisfying that the bundles in each family are pairwise non-isomorphic and their pull-backs
to $Y_{n}$ are trivial bundles.
\end{theorem}

\begin{proof}
Choose a positive integer $a$ such that $$(\prod_{p\leq n}p)\mid a,$$ where $\prod$ runs over all primes
$p\leq n$. Let $a_0=a_1=\cdots=a_{n}=a$. Then \[n!\mid a^{n+1}=a_0a_1\cdots a_{n}.\]

Denote by $V$ the set of homogeneous polynomials of degree $a$. It is a linear vector space. Write
$\Gr_{n+1}(V)$ for the Grassmannian variety of $(n+1)$-dimensional subspaces of $V$. Define
\[Z=\{\span_{\bbC}\{f_0,\dots,f_{n}\}\in\Gr_{n+1}(V):\rad(f_0,\dots,f_{n})\neq (x_0,\dots,x_{n})\}\] and
\[Z'=\{([W],[(z_0,\dots,z_n)])\in\Gr_{n+1}(V)\times\mathbb{P}^{n}:g(a_0,\dots,a_{n})=0, \forall g\in W\}.\]
Then $Z$ is the image of $Z'$ under the projection of $\Gr_{n+1}(V)\times\mathbb{P}^{n}$ to its first
component. Since $Z'$ is Zariski closed and the above projection map is proper, $Z$ is a Zariski closed
subset of $\Gr_{n+1}(V)$. On the other hand, it is clear that $Z$ is a proper subset of $\Gr_{n+1}(V)$.
Thus its complement is an open dense subset of $\Gr_{n+1}(V)$.

For a sequence $f=(f_0,f_1,\dots,f_{n})$ such that \[\span_{\bbC}\{f_0,\dots,f_{n}\}\in\Gr_{n+1}-Z,\]
let $E_{f}$ be the rank $n$ bundle on $\mathbb{P}^{n}$ defined by $f$ as in Proposition \ref{P:e=1}. We
know that, $E_{f}\cong E_{f'}$ if and only if $f$ differs from $f'$ by an invertible linear transformation,
i.e., they correspond to the same point in $\Gr_{n+1}-Z$. In this way we get a continuous family
$\{E_{f}|\ f\in\Gr_{n+1}(V)-Z\}$ of pairwise non-isomorphic rank $n$ algebraic vector bundles on
$\mathbb{P}^{n}$. A simple dimension counting shows that the dimension of $\Gr_{n+1}(V)-Z$ tends to infinity as
$a$ approches infinity. By Theorem \ref{T:Horrocks2} and Lemma \ref{L:indecomposable},
$\{\pi^{\ast}(E_{f})|\ f\in\Gr_{n+1}(V)-Z\}$ are pairwise non-isomorphic algebraic vector bundles on
$X_{n}$. Moreover, by Theorem \ref{T:e=1} $(1)$ each $\rho^{\ast}(E_{f})$ is a trivial bundle on $Y_{n}$.
Therefore, we finish the proof of the theorem.
\end{proof}

\begin{remark}\label{R:family}
In the proof of Theorem \ref{T:family}, furthermore one can show that $Z$ is irreducible and
\[\dim Z=\Gr_{n+1}(V)-1.\]
\end{remark}

\section{A theorem of Swan}

In this section, we present a proof for Theorem \ref{T:Kumar-Swan}(2) communicated to the author by
Professor Swan in an email.

\begin{theorem}[Swan, \cite{Swan2}]
Write \[R=\bbC[x_0,x_1,...,x_n,y_0,y_1,...,y_n]/\langle x_0y_0+x_1y_1+\cdots+x_{n}y_{n}-1\rangle.\]
Let $\{f_{0},\dots,f_{n}\}$ be homogeneous polynomials of $x_0,...,x_{n}$ with
\[\rad(f_0,f_1,...,f_{n})=(x_0,x_1,...,x_{n}).\]
If $n!\nmid\prod \deg f_{i}$, then $(f_0,f_1,...,f_{n})\in R^{n+1}$ is not completable.
\end{theorem}

\begin{proof}

Let $i: S^{2n+1}\longrightarrow\bbC^{n+1}-\{0\}$ and $r: \bbC^{n+1}-\{0\}\longrightarrow S^{2n+1}$
be the natural inclusion and projection, respectively. They are homotopy inverse to each other. Let
\[f=(f_{0},\dots,f_{n}):\bbC^{n+1}-\{0\}\longrightarrow\bbC^{n+1}-\{0\}\] and \[g=r\circ f\circ i.\]
By \cite{Swan-Towber}, it suffices to show that the map $g:S^{2n+1}\longrightarrow S^{2n+1}$ has
degree $\prod \deg f_{i}$.

One knows that the homology group $H_{2n+1}(\bbC^{n+1}-\{0\},\bbZ)$ is isomorphic to $\bbZ$. We define the
degree of the map $f$ by the scalar of the multiplication of its induced action on
$H_{2n+1}(\bbC^{n+1}-\{0\},\bbZ)$. It is clear that the degrees of $f$ and $g$ are equal. For maps
from $\bbC^{n+1}-\{0\}$ to itself of the form $f=(f_0,\dots,f_{n})$, it is clear that the degree is
multiplicative for the composition of maps. After composing $f$ with a map of the form
$(z_{0}^{a_0},...,z_{n}^{a_n})$, we may assume that $f_i$ are of the same degree $d$.

Consider the projective space $\mathbb{P}^{n+1}$ with homogeneous coordinates $[w:z_o:\dots:z_n]$.
Let $a\in S^{2n+1}\subset\bbC^{n+1}$ be a regular value of $f$. Define $X$ by the set of points such
that \[f_i(z)-a_iw^{d}=0,\ i=0,1,\dots,n.\] Then $X$ is not empty by the dimension theorem. However
the intersection of $X$ with $\{[w:z_o:\dots:z_n]\in\bbC^{n+1}|\ w = 0\}$ is empty by the unimodularity
assumption. Therefore we have $\dim X = 0$. This means that $X$ is a finite set. By Bezout's theorem,
$\deg X=\prod \deg f_{i}$. Now $X$ lies in $\{[w:z_o:\dots:z_n]\in\bbC^{n+1}|\ w\neq 0\}$. Making $w=1$,
we get that $X$ is equal to $f^{-1}(a)$. As $a$ is a regular value of $f$, all points of $X$ will have
multiplicity 1 and therefore $X$ will have $\prod \deg f_{i} = d^{n+1}$ distinct points.
It is clear that the radial projection $r$ gives a bijection $f^{-1}(a)\longrightarrow g^{-1}(a)$.
Thus $g^{-1}(a)$ also has $d^{n+1}$ distinct points.

\begin{lemma}\label{L:Swan}
If $a\in S^{2n+1}$ is a regular point of $g$, then the Jacobian of $g$ at each point $p\in g^{-1}(a)$
is positive and $a$ is also a regular value of $f$.
\end{lemma}

Lemma \ref{L:Swan} indicates that $f$ has a regular value $a\in S^{2n+1}$. For such an $a$, Lemma
\ref{L:Swan} indicates that the number of points in $g^{-1}(a)$ is equal to the degree of $g$. By the
above argument, the number of points in $g^{-1}(a)$ is equal to $d^{n+1}$. Therefore
$\deg g=d^{n+1}$.
\end{proof}

\begin{proof}[Proof of Lemma \ref{L:Swan}]
Given a point $p\in g^{-1}(a)$, changing $f$ to some $f'(z)=\lambda f(z)$ for some positive real number $\lambda$
if necessary, we may assume that $f(p)=a$ (i.e., $p\in f^{-1}(a)$). The tangent spaces of $\bbC^{n+1}-\{0\}$ at
$p$ and $a$ admit decompositions $T_p=N+S$, $T_a=N'+S'$, where $N,N'$ are the $1$-dimensional subspaces of normal
vectors, and $S,S'$ are the tangent spaces of $S^{2n+1}$ at $p,a$. The tangent map $i_{\ast,p}$ is an injective
map with image $S$, and the tangent map $r_{\ast,a}$ is a surjective map with kernel $N'$. Since $f$ is homogeneous
of degree $d$, we have $f_{\ast,p}(N)=N'$ and it is a positive scalar multiplication (here we identify $N,N'$ using
a non-zero normal vector field on $S^{2n+1}\subset\bbC^{n+1}-{0}$). Thus $f_{\ast,p}$ is a triangular matrix with
$f_{\ast,p}|_{N}$ and $g_{\ast,p}$ the block diagonal parts. The linear map $f_{\ast,p}|_{N}$ is clearly a positive
scalar multiplication. Moreover, since $f$ is holomorphic, we have $\det(f_{\ast,p})\geq 0$. Hence
$\det(g_{\ast,p})\geq 0$. Therefore $\det(g_{\ast,p})>0$ since $a$ is a regular value of $g$. By this we get
$\det(f_{\ast,p})>0$. Similarly, we can show that $\det(f_{\ast,q})>0$ for any other $q\in f^{-1}(a)$.
Therefore $a$ is a regular value of $f$.
\end{proof}

\section{Some remarks}

Based on our study of algebraic vector bundles on $\mathbb{P}^{n}$ with $e(E)=1$ and their pull-backs
to $X_{n}=\bbC^{n+1}\backslash\{0\}$ and $Y_{n}=\SL_{n+1}/\SL_{n}$, we ask the following questions.

\begin{question}\label{Q1}
Given an algebraic vector bundle $E$ on $\mathbb{P}^{n}$,
\begin{itemize}
\item[(1)]{is $\rho^{\ast}E$ a trivial bundle on $X_{n}$ whenever $\rank E\geq n+1$?}


\item[(2)]{In the case of $e(E)=1$, when is $E$ stable or semi-stable?}

\item[(3)]{In the case that $E$ is non-split and has rank at most $n-1$, can it satisfy that $M_{i}(E)=0$ for
$2\leq i\leq n-2$?}
\end{itemize}
\end{question}

A theorem of Kumar-Peterson-Rao (cf. \cite{Kumar-Peterson-Rao}) confirms the non-existence in Question
\ref{Q1} (3) in the case that $n$ is even. In the case that $n$ is odd, it implies that the rank is at least $n-1$.

\begin{theorem}[\cite{Kumar-Peterson-Rao}]\label{T:rank}
If $E$ is a non-split algebraic vector bundle on $\mathbb{P}^n$ with $M_{i}(E)=0$ for $2\leq i\leq n-2$,
then $\rank E\geq 2[\frac{n}{2}]$.
\end{theorem}

The 3-term resolution we constructed for bundles $E$ on $\mathbb{P}^{n}$ with $M_{i}(E)=0$
for $1\leq i\leq n-2$ is a special case of a general resolution theorem.

\begin{proposition}[Three-term resolution]\label{P:syzygy}
Given an algebraic vector bundle $E$ of rank $r$ on $P=\mathbb{P}^n$, we have a (canonical) exact sequence
of vector bundles \[0\longrightarrow E'\longrightarrow E''=\bigoplus_{1\leq i\leq e+r}
\mathcal{O}_{\mathbb{P}^n}(b_{i})\longrightarrow E\rightarrow 0\] such that $e=e(E)$ and $H^{0}(E''(k))
\longrightarrow H^{0}(E(k))$ is surjective for $k\in\bbZ$. For such an exact sequence, we have $M_{1}(E')=0$
and $H^{i}(E'(k))\cong H^{i-1}(E(k))$ for $2\leq i\leq n-1$ and $k\in\bbZ$.
\end{proposition}

By Proposition \ref{P:syzygy} and Theorem \ref{T:Horrocks}, we have the following resolution.

\begin{proposition}[Syzygy resolution]\label{P2:syzygy}
Given an algebraic vector bundle $E$ of rank $r$ on $P=\mathbb{P}^n$, there is a (canonical) exact
sequence of vector bundles \[0\longrightarrow E_{n}\stackrel{\phi_{n}}\longrightarrow \cdots
\longrightarrow E_1\stackrel{\phi_1}\longrightarrow E=E_0\longrightarrow 0\]
such that each $E_{i}$ ($i\geq 1)$ is a direct sum of line bundles, $\rank\ker\phi_{i}=e(\Im\phi_{i})$
for $1\leq i\leq n$, and $H^{0}(E_{i}(k))\longrightarrow H^{0}(\Im\phi_{i}(k))$ is surjective for
$1\leq i\leq n$ and $k\in\bbZ$.
\end{proposition}

\begin{definition}
We call \[s(E)=\max\{i|E_{i}\neq 0\}-1\] the {\bf complexity} of $E$.
\end{definition}

\begin{remark}
By Proposition \ref{P:syzygy}, the complexity $s(E)$ is equal to
\[\min\{i|M_{j}(E)=0,\forall j, 1\leq j\leq n-i\}-1.\]
\end{remark}



\end{document}